\documentclass[11pt]{amsart}

\usepackage{amsfonts,amssymb,amsthm,eucal,color,bbm,verbatim,enumitem,graphicx}
\usepackage[french,english]{babel}

\usepackage[margin=1.25in]{geometry}

\newtheorem{thm}{Theorem}

\newtheorem{lemma}[thm]{Lemma}
\newtheorem{prop}[thm]{Proposition}

\newcommand{\E}{\mathbb{E}}

\newcommand{\Prob}{\mathbb{P}}
\newcommand{\N}{\mathbb{N}}
\newcommand{\n}{\mathcal{N}}

\newcommand{\C}{\mathbb{C}}

\DeclareMathOperator{\tr}{tr}

\newcommand{\abs}[1]{\left\vert #1 \right\vert}

\newcommand{\eps}{\varepsilon}

\DeclareMathOperator{\var}{Var}

\newcommand{\Set}[2]{\left\{#1 \mathrel{} \middle| \mathrel{} #2 \right\}}

\allowdisplaybreaks[3]

\author{Elizabeth S.\ Meckes and Mark W.\ Meckes}

\address{Department of Mathematics, Case Western Reserve University,
10900 Euclid Ave., Cleveland, Ohio 44106, U.S.A.}

\email{elizabeth.meckes@case.edu}

\address{Department of Mathematics, Case Western Reserve University,
10900 Euclid Ave., Cleveland, Ohio 44106, U.S.A.}

\email{mark.meckes@case.edu}

\title{A rate of convergence for the circular law for the complex Ginibre ensemble}

\begin{document}

\maketitle

\begin{abstract}
  We prove rates of convergence for the circular law for the complex
  Ginibre ensemble.  Specifically, we bound the $L_p$-Wasserstein
  distances between the empirical spectral measure of the normalized
  complex Ginibre ensemble and the uniform measure on the unit disc,
  both in expectation and almost surely.  For $1 \le p \le 2$, the
  bounds are of the order $n^{-1/4}$, up to logarithmic factors.
\end{abstract}

\begin{otherlanguage}{french}
\begin{abstract}
Nous \'etablissons des vitesses de convergence pour la loi du cercle de l'ensemble de Ginibre complexe.  Plus pr\'ecis\'ement, nous donnons des bornes sup\'erieurs pour les distances de Wasserstein d'ordre $p$ entre la mesure spectrale empirique de l'ensemble de Ginibre complexe normalis\'ee et la mesure uniform du disque, dans l'esp\'erance et presque s\^urement.  Si $1\le p\le 2$, les bournes sont de la taille $n^{-1/4}$, \`a des facteurs logarithmiques. 
\end{abstract}
\end{otherlanguage}

\section{Introduction}
\label{S:intro}

Consider an $n \times n$ random
matrix $G_n$ with i.i.d.\ standard complex Gaussian entries; put slightly differently,
$G_n$ is a random element of the set of $n\times n$ matrices over
$\C$, whose distribution has density proportional to $e^
{-\tr(GG^*)}$.  The random matrix $G_n$ is said
to belong to the \emph{complex Ginibre ensemble}. Although this ensemble was
introduced by Ginibre \cite{Ginibre} without any particular
application in mind, the eigenvalues of $G_n$ have since been used to
model a wide variety of physical phenomena; see in particular
\cite{Fo} and \cite{KhSo} for
references.

The central result about the asymptotic behavior of the eigenvalues of
$G_n$ is the famous circular law. Let $\mu_n$ denote the empirical
spectral measure of $\frac{1}{\sqrt{n}} G_n$; that is,
\[
\mu_n = \frac{1}{n} \sum_{k=1}^n \delta_{\lambda_k},
\]
where $\lambda_1, \dots, \lambda_n$ are the eigenvalues of
$\frac{1}{\sqrt{n}} G_n$. The circular law states that when $n \to
\infty$, $\mu_n$ converges in some sense to the uniform measure $\nu$
on the unit disc $D:=\Set{z \in \C}{\abs{z} \le 1}$.  This was first
established by Mehta \cite{Mehta-1}, who showed that the mean
empirical spectral measure $\E \mu_n$ converges weakly to $\nu$.  A
large literature followed, which established the circular
law for more general random matrix ensembles, and for stronger forms
of convergence, culminating in the recent proof by Tao and Vu
\cite{TaVu2} of the circular law for random matrices with i.i.d.\
entries with arbitrary entries with finite variance, in the sense of
almost sure weak convergence.  The reader is referred to the survey by
Bordenave and Chafa\"i \cite{BoCh} for further history and related
results.

The main results of this paper give rates of convergence for the
circular law for the complex Ginibre ensemble $G_n$, both in
expectation and almost surely.

\begin{thm}\label{T:expected-Wp-distance-intro}
  There is a constant $C > 0$ such that for all $n \in \N$ and all
  $p\ge 1$,
  \[
  \E W_p(\mu_n,\nu)\le C \max \left\{ \frac{\sqrt{p}}{n^{1/4}},
    \left(\frac{\log n}{n}\right)^{\frac{1}{2p}} \right\},
  \]
  where $W_p(\mu,\nu)$ denotes the $L_p$-Wasserstein distance between
  probability measures $\mu$ and $\nu$.
\end{thm}

In particular, in the most widely used Wasserstein metrics, namely
$p=1,2$, we have
\[
\E W_1(\mu_n,\nu)\le\frac{C}{n^{1/4}}\qquad
\mbox{and}\qquad 
\E W_2(\mu_n,\nu)\le C\left(\frac{\log n}{n}\right)^{\frac{1}{4}}.
\]

\begin{thm}\label{T:as-rate}
  For each $p \ge 1$ there is a constant $K_p > 0$ such that with
  probability $1$, for sufficiently large $n$,
  \[
  W_p(\mu_n, \nu) \le K_p \frac{\sqrt{\log n}}{n^{1/4}}
  \]
  when $1 \le p \le 2$, and
  \[
  W_p(\mu_n, \nu) \le K_p \left(\frac{\log n}{n}\right)^{1/2p}
  \]
  when $p > 2$.
\end{thm}

Recall that for any $p \ge 1$, the $L_p$-Wasserstein distance between two
probability measures $\mu$ and $\nu$ on $\C$ is defined by
\[
W_p(\mu, \nu) = \left(\inf_{\pi \in \Pi(\mu, \nu)} \int \abs{w - z}^p
  \ d\pi(w, z) \right)^{1/p},
\]
where $\Pi(\mu, \nu)$ is the set of all couplings of $\mu$ and $\nu$;
i.e., probability measures on $\C \times \C$ with marginals $\mu$ and
$\nu$ (see, e.g., \cite{Vi}).

A few related results have appeared previously.  In \cite[Section
14]{TaVu1}, Tao and Vu sketched an argument giving an almost sure
convergence rate for the empirical spectral measure of a random matrix
$\frac{1}{\sqrt{n}}M_n$ with i.i.d.\ entries with a finite moment of order $2+\eps$.
The convergence in this case was in Kolmogorov distance (sup-distance
between bivariate cumulative distribution functions), and the rate is
of order $n^{-c}$ for some unspecified (but rather small) $c = c(\eps)
> 0$.  Earlier, Bai \cite{Bai} established, as an intermediate
technical tool, a convergence rate for the empirical spectral measures
of the Hermitianized random matrices $(M_n - z I_n)^* (M_n - z I_n)$.

In a different direction, Sandier and Serfaty \cite{SaSe} and Rougerie
and Serfaty \cite{RoSe} studied empirical measures of Coulomb gases,
which for particular values of certain parameters have the same
distribution as $\mu_n$.  Among their results are tail bounds for
distances between these measures from deterministic equilibrium
measures, in terms of metrics which are dual to Sobolev norms on a
ball.  For a certain choice of parameter, in the $2$-dimensional case
their metric becomes
\[
\sup\Set{ \int_{rD} f \ d\mu(z) - \int_{rD} f \ d\nu(z)}{f \text{
    $1$-Lipschitz}};
\]
without the restriction to $rD$ this would coincide with $W_1$, by the
Kantorovitch duality theorem.

\bigskip

The basic idea of the proofs of Theorems
\ref{T:expected-Wp-distance-intro} and \ref{T:as-rate} is reasonably
simple, but verifying all of the details gets somewhat technical, and
so we first give an outline of our approach.
\begin{enumerate}[label = \bfseries Step \arabic*:]
\item We begin by \textbf{ordering the eigenvalues
    $\{\lambda_k\}_{k=1}^n$ in a spiral fashion}.  Specifically, we
  define a linear order $\prec$ on $\C$ by making $0$ initial, and for
  nonzero $w, z \in \C$, we declare $w \prec z$ if any of the
  following holds:
  \begin{itemize}
  \item $\lfloor \sqrt{n} \abs{w} \rfloor < \lfloor \sqrt{n} \abs{z} \rfloor$.
  \item $\lfloor \sqrt{n} \abs{w} \rfloor = \lfloor \sqrt{n} \abs{z} \rfloor$ and
    $\arg w < \arg z$.
  \item $\lfloor \sqrt{n} \abs{w} \rfloor = \lfloor \sqrt{n} \abs{z} \rfloor$,
    $\arg w = \arg z$, and $\abs{w} \ge \abs{z}$.
  \end{itemize}
  Here we are using the convention that $\arg z \in (0,2\pi]$.   
  
  We order the eigenvalues according to $\prec$: first the eigenvalues
  in the disc of radius $\frac{1}{\sqrt{n}}$ are listed in order of
  increasing argument, then the ones in the annulus with inner radius
  $\frac{1}{\sqrt{n}}$ and outer radius $\frac{2}{\sqrt{n}}$ in order
  of increasing argument, and so on.  (With probability $1$, no two
  eigenvalues of $G_n$ have the same argument; thus the details of the
  last condition in the definition of $\prec$ are irrelevant and it is
  included only for completeness.)

\item We \textbf{define \emph{predicted locations} for (most of) the
    eigenvalues} as follows.  Fix some $m$ so that $n-m$ is a perfect
  square.  Then $\tilde{\lambda}_1 = 0$, $\{\tilde{\lambda}_2,
  \tilde{\lambda}_3, \tilde{\lambda}_4\}$ are $\frac{1}{\sqrt{n}}$
  times the $3^{\mathrm{rd}}$ roots of unity (in increasing order with
  respect to $\prec$), the next five are $\frac{2}{\sqrt{n}}$ times
  the $5^{\mathrm{th}}$ roots of unity, and so on until
  $\tilde{\lambda}_{n-m}$.

  Formally, given $1 \le k \le n-m$, write $\ell = \lceil \sqrt{k}
  \rceil$ and $q = k - (\ell - 1)^2$, so that
  \begin{equation} 
    \label{E:klq} 
    k = (\ell-1)^2 + q \qquad \text{and}
    \qquad 1 \le q \le 2\ell -1.
  \end{equation}
  Now define
  \[
  \tilde{\lambda}_k = \frac{\ell-1}{\sqrt{n}} e^{2 \pi i q / (2\ell - 1)}.
  \]
  Observe that the sequence
  $\bigl(\tilde{\lambda}_k\bigr)_{k=1}^{n-m}$ is increasing with
  respect to $\prec$.

\item \label{Step:ew-concentration} We \textbf{show that most of the
    eigenvalues $\lambda_k$ concentrate around their predicted
    locations} $\tilde{\lambda}_k$.  The eigenvalue process of $G_n$
  is a determinantal point process, from which concentration
  inequalities for the number of eigenvalues within subsets of $D$
  follow.  We apply this concentration property to the number of
  eigenvalues in an initial segment with respect to the order $\prec$.
  Geometric arguments allow one to move from this concentration to
  concentration of individual eigenvalues around their predicted
  values.

\item We \textbf{couple the empirical spectral measure $\mu_n$ to the
    measure $\nu_n$} which puts mass $\frac{1}{n}$ at each point
  $\tilde{\lambda}_1, \dots, \tilde{\lambda}_{n-m}$, and mass
  $\frac{m}{n}$ uniformly on the annulus \\$\Set{ z \in \C }{\sqrt{1 -
      \frac{m}{n}} \le \abs{z} \le 1 }$.  The concentration
  established in the previous step allows us to estimate $W_p(\mu_n,
  \nu_n)$ via this coupling.

\item \label{Step:coupling-tildes-to-uniform}The measure
  \textbf{$\nu_n$ is approximately uniform} on $D$.

\end{enumerate}

This approach adapts those taken by Dallaporta \cite{Da12} for the
Gaussian Unitary Ensemble, and by the authors
\cite{MeMe-powers} for random unitary matrices.  In those settings,
the linear order of the eigenvalues was of critical importance.  The lack of a
natural order on the complex plane is the major obstacle in adapting
the methods of \cite{Da12,MeMe-powers} for the Ginibre ensemble, and
it is this difficulty which is addressed by the introduction of the
spiral order $\prec$.

The rest of this paper is organized as follows.  In Section
\ref{S:technical-tools} we dispense with Step 5 of the outline, and
collect the main technical tools which will be used in the rest of the
paper.  In Section \ref{S:means-and-variances} we estimate the mean
and variance of the number of eigenvalues in an initial segment with
respect to the order $\prec$. In Section \ref{S:deviations}, we derive estimates for the concentration of individual eigenvalues around their predicted values (Step 3 above), using the results of the previous two
sections.  Finally, in Section \ref{S:distances}, we carry out the coupling argument (Step 4
of the outline) and complete the proofs of Theorems
\ref{T:expected-Wp-distance-intro} and \ref{T:as-rate}.  We also observe (Theorem \ref{T:TV-mean}) that our results yield the correct
rate of convergence of the mean empirical spectral measure in the
total variation metric.

\section{Technical tools}
\label{S:technical-tools}

We begin by taking care of Step 5 in the outline above.  Recall that
$\nu_n$ is the measure which puts mass $\frac{1}{n}$ at each point
$\tilde{\lambda}_1, \dots, \tilde{\lambda}_{n-m}$, and mass
$\frac{m}{n}$ uniformly on the annulus $\Set{ z \in \C }{\sqrt{1 -
    \frac{m}{n}} \le \abs{z} \le 1 }$.

\begin{lemma}
  \label{T:coupling}
  For each positive integer $n$ and each $p \ge 1$,
  \(
  W_p(\nu_n, \nu) < \frac{8}{\sqrt{n}}.
  \)
\end{lemma}

\begin{proof}
  We couple $\nu_n$ to $\nu$ as follows.  The sector
  \[
  S_k:=\Set{ z \in \C }{ \frac{\ell-1}{\sqrt{n}} \le \abs{z} <
    \frac{\ell}{\sqrt{n}}, \, 
    \frac{2\pi  (q-1)}{2\ell - 1} \le \arg z \le \frac{ 2\pi q}{2\ell
      - 1} },
  \]
  where $k$, $\ell$, and $q$ are related by \eqref{E:klq}, satisfies
  $\nu(S_k) = 1/n$ for each $1 \le k \le n-m$.  All of the mass in
  $S_k$ is coupled to $\tilde{\lambda}_k$, and the identity coupling
  is used in the annulus $\Set{ z \in \C }{ \sqrt{1 - \frac{m}{n}} \le
    \abs{z} \le 1 }$.  If $r \in \left[ \frac{\ell}{\sqrt{n}},
    \frac{\ell-1}{\sqrt{n}}\right]$ and $\varphi \in \left[\frac{2\pi
      (q-1)}{2\ell - 1}, \frac{ 2\pi q}{2\ell - 1}\right]$, then
  \begin{equation*}
    \begin{split}
      \abs{r e^{i \varphi} - \tilde{\lambda}_k} 
      & \le \abs{r e^{i \varphi} - r e^{2 \pi i q / (2\ell - 1)}} +
      \abs{r e^{2 \pi i q / (2\ell - 1)} - \tilde{\lambda}_k} \\
      & \le r \abs{\varphi -  \frac{2 \pi q}{(2\ell - 1)}} + \abs{r - \frac{\ell-1}{\sqrt{n}}} \\
      & \le \frac{2\pi \ell}{(2\ell - 1) \sqrt{n}} + \frac{1}{\sqrt{n}}
      < \frac{8}{\sqrt{n}}.
    \end{split}
  \end{equation*}
  Therefore
  \[
  W_p(\nu_n, \nu) < \left(\frac{n-m}{n}
    \left(\frac{8}{\sqrt{n}}\right)^p + \frac{m}{n} \big(0\big)\right)^{1/p} \le
  \frac{8}{\sqrt{n}}.
  \qedhere
  \]
\end{proof}

Lemma \ref{T:coupling} shows that, up to the constant $8$, $\nu_n$ is
an optimal approximation of $\nu$ by an empirical measure on $n$
points.  Indeed, suppose that $x_1, \dots, x_n$ are any $n$ points in
$\C$, and let $\rho_n = \frac{1}{n} \sum_{i=1}^n \delta_{x_i}$.  Then
the area of the union of the $\eps$-discs centered at the $x_i$ is at
most $n \pi \eps^2$, so $W_1(\rho_n, \nu) \ge (1 - n \eps^2) \eps$,
since a fraction at least $(1-n\eps^2)$ of the mass of $\nu$ must move
a distance at least $\eps$ in transporting $\nu$ to $\rho_n$.
Optimizing in $\eps$ gives $W_p(\rho_n, \nu) \ge W_1(\rho_n, \nu) \ge
\frac{2}{3\sqrt{3n}}$.

\medskip

\begin{prop}
  \label{T:ew-counting-deviations}
  Let $A\subseteq D$ be measurable, and let $\n (A)$ denote the number
  of eigenvalues of $\frac{1}{\sqrt{n}}G_n$ lying in $A$.  Then
  \[
  \Prob \left[ \n(A) - \E \n(A) \ge t \right] \le \exp \left[ - \min
    \left\{ \frac{t^2}{4 \sigma^2}, \frac{t}{2}\right\}\right]
  \]
  and 
  \[
  \Prob \left[ \E \n(A) - \n(A) \ge t \right] \le \exp \left[ - \min
    \left\{ \frac{t^2}{4 \sigma^2}, \frac{t}{2}\right\}\right]
  \]
  for each $t \ge 0$, where $\sigma^2 = \var \n(A)$.
\end{prop}

\begin{proof} 
  The eigenvalues of $G_n$ form a determinantal point process on $\C$
  with the kernel
  \begin{equation} \label{E:DPP}
    \begin{split}
      K(z,w) & = \frac{1}{\pi} e^{-(\abs{z}^2 + \abs{w}^2) / 2}
      \sum_{k=0}^{n-1} \frac{(z\overline{w})^k}{k!} \\
      & = \frac{1}{\pi} e^{-\abs{z-w}^2/2} \left(1 -
        e^{-z\overline{w}} \sum_{k=n}^\infty \frac{(z
          \overline{w})^k}{k!} \right).
    \end{split}
  \end{equation}
  The reader is referred to \cite{HKPV06} for the definition of a
  determinantal point process.  The fact that the eigenvalues of $G_n$
  form such a process follows from the original work of Ginibre
  \cite{Ginibre}; see also \cite[Chapter 15]{Mehta-3}.

  This fact combines crucially with \cite[Theorem 7]{HKPV06} (see also
  \cite[Corollary 4.2.24]{AGZ}), which says that $\n(A)$ is
  distributed exactly as a sum of independent $\{0,1\}$-valued random
  variables, whose parameters are given by the eigenvalues (which
  necessarily lie in $[0,1]$) of the
  operator $\mathcal{K}$ on $L_2(A)$ defined by 
\[\mathcal{K}f(z)=\int_A K(z,w)f(w)\ dw.\]
  The statement of the proposition is thus simply Bernstein's classical tail
  inequality for sums of independent bounded random variables
  (see, e.g., \cite[Lemma 2.7.1]{Talagrand}).
\end{proof}

As discussed in Step 3 of the outline, to bound the deviations of
$\lambda_k$ about its predicted location $\tilde{\lambda}_k$, we first
use Proposition \ref{T:ew-counting-deviations} to bound the deviations
of the counting functions for initial segments with respect to the
order $\prec$.  Specifically, we will consider $\n(A_{j,\theta})$, where
\begin{equation*}
  \begin{split}
    A_{j,\theta} &:= \Set{ z \in \C}{ z \prec \frac{j}{\sqrt{n}} e^{i
        \theta}}\\
    &= \Set{z\in\C }{ \abs{z} < \frac{j}{\sqrt{n}} } \cup
    \Set{ z \in \C }{ \frac{j}{\sqrt{n}} \le \abs{z} < \frac{j+1}{\sqrt{n}},\
      0 < \arg z  \le \theta },
     \end{split}
\end{equation*}
for $1 \le j \le \sqrt{n} - 1$ and $0 < \theta \le 2\pi$ (see Figure
\ref{F:segment}).

\begin{figure} 
  \begin{centering}
    \includegraphics[width=3in]{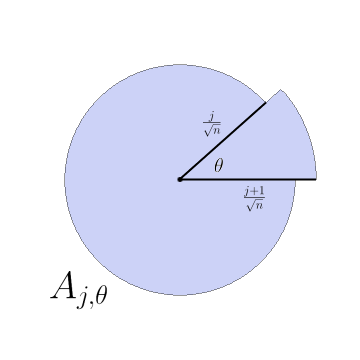}
  \end{centering}
  \caption{}
  \label{F:segment}
\end{figure}

\medskip

Finally, we conclude this section by collecting a few known formulas and estimates
which will be used repeatedly below.  The following integral formula
can be proved by repeated integration by parts; we omit the proof.

\begin{lemma}\label{T:gamma-integral}
  If $k$ is a nonnegative integer and $a > 0$, then
  \[
  \frac{1}{k!} \int_{a}^\infty s^k e^{-s} \ ds = e^{-a} \sum_{\ell =
    0}^k \frac{a^\ell}{\ell!},
  \]
  and consequently
  \[
  \frac{1}{k!} \int_{0}^a s^k e^{-s} \ ds = e^{-a} \sum_{\ell =
    k+1}^\infty \frac{a^\ell}{\ell!}.
  \]
\end{lemma}

The following inequality follows from a standard Chernoff bound
argument for Poisson random variables.

\begin{lemma}\label{T:Poisson-tail}
  If $0 < \lambda \le n$, then
  \(
  \sum_{k=n}^{\infty} \frac{\lambda^k}{k!} \le \left(\frac{e\lambda}{n}\right)^n.
  \)
\end{lemma}

\begin{proof}
  Let $X$ have a Poisson distribution with parameter $\lambda$.
  Assuming for simplicity that $\lambda < n$, let $t = \log
  (n/\lambda) > 0$.  Then 
  \begin{equation*}
      \sum_{k=n}^\infty \frac{\lambda^k}{k!} = e^\lambda \Prob
      \left[ X \ge n \right] \le e^{\lambda - tn} \E e^{tX} =
      e^{\lambda e^t - tn}
      =  \left(\frac{e\lambda}{n}\right)^n.
      \qedhere
  \end{equation*}
\end{proof}

Finally, we will use the following uniform version of Stirling's
approximation.

\begin{lemma}\label{T:Stirling}
  For each positive integer $n$,
  \(
  \sqrt{2\pi} n^{n+\frac{1}{2}} e^{-n} \le n! \le e n^{n+\frac{1}{2}} e^{-n}.
  \)
\end{lemma}

\begin{proof}
  The following version of Stirling's approximation appears as
  \cite[(9.15)]{Feller}:
  \[
  \sqrt{2\pi} n^{n+\frac{1}{2}} e^{-n+\frac{1}{12n+1}} < n! < \sqrt{2\pi} n^{n+\frac{1}{2}}
  e^{-n+\frac{1}{12n}}.
  \]
  The lemma is trivially true when $n=1$, and for $n \ge 2$, the lemma
  follows since $\sqrt{2\pi} e^{1/12n} \le \sqrt{2\pi} e^{1/24} < e$.
\end{proof}

\section{Means and variances}
\label{S:means-and-variances}

Concentration inequalities for the random variables
$\n\left(A_{j,\theta}\right)$ about their means follow from
Proposition \ref{T:ew-counting-deviations}, but in order to make use of
them, fairly sharp estimates on the means and variances of the
$\n\left(A_{j,\theta}\right)$ are needed.  These estimates, like the
proof of Proposition \ref{T:ew-counting-deviations}, make use of the
determinantal point process structure of the eigenvalues of $G_n$.

\begin{prop}\label{T:means}
  If $A \subseteq D$, 
  \[
  \frac{n\abs{A}}{\pi} - e\sqrt{n}\le 
  \E \n(A) \le \frac{n\abs{A}}{\pi},
  \]
  where $\abs{A}$ denotes the area of $A$.  Moreover, if $A \subseteq
  \left(1 - \sqrt{\frac{\log n}{n}} \right)D$, then
  \[
  \frac{n\abs{A}}{\pi} -  e^2 \le \E \n(A) \le \frac{n\abs{A}}{\pi}.
  \]
\end{prop}

\begin{proof}
  The determinantal point process structure of the eigenvalues of
  $G_n$ implies that that $\E \n(A) = \int_{\sqrt{n} A} K(z,z) \ dz$
  (where $dz$ denotes integration with respect to Lebesgue measure on
  $\C$), so that
  \[
    \E \n(A) 
       = \frac{1}{\pi } \int_{\sqrt{n} A} \left(1 - \sum_{k=n}^\infty
      e^{-\abs{z}^2} \frac{\abs{z}^{2k}}{k!}  \right) \ dz
    = \frac{n\abs{A}}{\pi} - \frac{1}{\pi }\int_{\sqrt{n} A} \sum_{k=n}^\infty
    e^{-\abs{z}^2} \frac{\abs{z}^{2k}}{k!} \ dz.
  \]
  Using Lemma
  \ref{T:Poisson-tail} and then integrating in polar coordinates,
  \begin{align*}
    \frac{1}{\pi} \int_{\sqrt{n} A} \sum_{k=n}^\infty e^{-\abs{z}^2}
    \frac{\abs{z}^{2k}}{k!} \ dz
    & \le \left(\frac{e}{n}\right)^n \int_{\sqrt{n} D} e^{-\abs{z}^2}
    \abs{z}^{2n} \ dz \\
    & = 2\left(\frac{e}{n}\right)^n \int_0^{\sqrt{n}} e^{-r^2}
    r^{2n+1} \ dr 
     < \left(\frac{e}{n}\right)^n n! \le \pi e \sqrt{n},
  \end{align*}
  by Stirling's approximation (Lemma \ref{T:Stirling}). 
  
  If $A \subseteq rD$ for $r \le 1$ then, using Lemmas
  \ref{T:Poisson-tail}, \ref{T:gamma-integral}, and 
  \ref{T:Poisson-tail} again, 
  \begin{equation*}\begin{split}
      \frac{1}{\pi} \int_{\sqrt{n} rD} \sum_{k=n}^\infty
      e^{-\abs{z}^2} \frac{\abs{z}^{2k}}{k!} \ dz
      &\le \left(\frac{e}{n}\right)^n \int_0^{r^2n} e^{-s} s^n \ ds \\
      &= e^{-r^2n} \left(\frac{e}{n}\right)^n n!\sum_{\ell=n+1}^\infty
      \frac{(r^2n)^\ell}{\ell!} \\&\le  e^{-r^2n} e \sqrt{n} (e r^2)^n 
       \le e^2 \sqrt{n} e^{-n(1-r^2)^2/2},
    \end{split}\end{equation*}
  since $\log(1-\eps) \le -\eps - \eps^2/2$ for $0 < \eps < 1$.
  Finally, let $r = 1 - \sqrt{\frac{\log n}{n}}$.  Then \( \sqrt{n}
  e^{-n(1-r^2)^2/2} \le \sqrt{n} e^{-n(1-r)^2/2} = 1.  \)
\end{proof}
  
We will also need estimates for the expected number of eigenvalues
outside of discs of radius $R\ge 1$.
\begin{prop}
  \label{T:outside-disc}
  For each $R \ge 1$,
  \(
  \E \n (\C \setminus RD) \le 
  \frac{1}{\sqrt{2\pi}}\sqrt{n} e^n R^{2(n-1)} e^{-nR^2} .
  \)
\end{prop}

\begin{proof}
  Again using the determinantal point process kernel in \eqref{E:DPP},
  \begin{align*}
          \E \n (\C \setminus RD)
      &       = \frac{1}{\pi} \sum_{k=0}^{n-1} \frac{1}{k!} \int_{\C \setminus
        \sqrt{n} RD} e^{-\abs{z}^2} \abs{z}^{2k} \ dz \\
      & = \sum_{k=0}^{n-1} \frac{1}{k!} \int_{nR^2}^\infty r^k e^{-r}
      \ dr
      \\& = e^{-nR^2} \sum_{k=0}^{n-1} \sum_{\ell = 0}^k
      \frac{(nR^2)^\ell}{\ell!} \\
      & = e^{-nR^2} \sum_{\ell = 0}^{n-1} \frac{(nR^2)^\ell}{\ell!}
      (n-\ell) \\
      & = n e^{-nR^2} \left( \sum_{\ell = 0}^{n-1}
        \frac{(nR^2)^\ell}{\ell!}
        - R^2 \sum_{\ell = 0}^{n-2} \frac{(nR^2)^\ell}{\ell!} \right)
      \\
      & = n e^{-nR^2} \left( \frac{(nR^2)^{n-1}}{(n-1)!} - (R^2 - 1)
        \sum_{\ell = 0}^{n-2} \frac{(nR^2)^{\ell}}{\ell!} \right) \\
      & \le n e^{-nR^2} \frac{(nR^2)^{n-1}}{(n-1)!}\\
      & \le \frac{1}{\sqrt{2\pi}} e^{-nR^2} \sqrt{n} e^n R^{2(n-1)}
      \end{align*}
  by Stirling's approximation.
\end{proof}

\begin{prop} \label{T:variances}
  For each $1 \le j \le \sqrt{n} - 1$ and $0 \le \theta \le 2\pi$,
  \[
  \var \n(A_{j, \theta}) \le 16 j.
  \]
\end{prop}

The constant $16$ in the statement of Proposition \ref{T:variances} is
not optimal and is included only for the sake of concreteness.

\begin{proof}
  By an argument in \cite[Appendix B]{Gu}, 
  \begin{equation}
    \label{E:variance}
    \begin{split}
      \var(\n(A_{j,\theta}))
      &= \int\limits_{\{\abs{z} \le j\}} \int\limits_{\{\abs{w} \ge j+1\}}
      \abs{K(z,w)}^2 \ dw \ dz \\
      & \quad + \int\limits_{\{\abs{z} \le j\}} \int\limits_{\{j\le \abs{w} \le j+1,\ \arg w \ge \theta\}} 
      \abs{K(z,w)}^2 \ dw \ d z \\
      & \quad + \int\limits_{\{j\le \abs{z} \le j+1,\ \arg z  \le \theta\}}
      \int\limits_{\{\abs{w} \ge j+1\}} \abs{K(z,w)}^2 \ dw \ dz \\
      & \quad + \int\limits_{\{j\le \abs{z} \le j+1,\ \arg z  \le \theta\}}
      \int\limits_{\{j \le \abs{w} \le j+1, \ \arg w \ge \theta\}} \abs{K(z,w)}^2 \ dw \ dz 
    \end{split} 
  \end{equation}
  Observe that
  \[
  \abs{K(r_1 e^{i\varphi_1}, r_2 e^{i \varphi_2})}^2 = \frac{1}{\pi^2}
  \sum_{k,\ell = 0}^{n-1} \frac{1}{k! \ell!}  e^{-(r_1^2 + r_2^2)}
  (r_1 r_2)^{k+ \ell} e^{i (k-\ell) (\varphi_1 - \varphi_2)}.
  \]

  Integrating in polar coordinates, the first integral in
  \eqref{E:variance} is
  \begin{align*}
    \frac{1}{\pi^2} & \sum_{k,\ell = 0}^{n-1} \frac{1}{k! \ell!}
    \int_0^{j} r^{k+\ell + 1} e^{-r^2} \ dr \int_{j+1}^\infty r^{k+\ell
      + 1} e^{-r^2} \ dr \int_0^{2\pi} e^{i \varphi(k-\ell)} \ d\varphi
    \int_0^{2\pi} e^{i \varphi(\ell-k)} \ d\varphi \\
    & = \sum_{k=0}^{n-1} \frac{1}{k!^2} \int_0^{j^2} s^k e^{-s} \ ds \int_{(j+1)^2}^\infty
    s^k e^{-s} \ ds \\
    & \le \sum_{k=0}^{j^2}\frac{1}{k!} \int_{j^2}^\infty
    s^k e^{-s}\ ds+ \sum_{k=(j+1)^2}^{n-1}\frac{1}{k!}\int_0^{(j+1)^2}
    s^k e^{-s}\ ds + (2j+1).
  \end{align*}
  Here we have used that the angular integrals are nonzero only if
  $k=\ell$, and that the integrals in the second line are bounded by
  $k!$. Note also that if $j^2<n-1<(j+1)^2$, the second term is not needed, and if $j^2\ge n-1$, then the second and third terms are not needed.  By Lemma \ref{T:gamma-integral} and Stirling's approximation,
  \begin{equation*}\begin{split}
      \sum_{k=0}^{j^2} \frac{1}{k!}\int_{j^2}^\infty s^k e^{-s}ds &=
      \sum_{k=0}^{j^2} e^{-j^2} \sum_{\ell=0}^k
      \frac{j^{2\ell}}{\ell!}
      =e^{-j^2}\sum_{\ell=0}^{j^2}\sum_{k=\ell}^{j^2}
      \frac{j^{2\ell}}{\ell!}  = e^{-j^2}\sum_{\ell=0}^{j^2}
      \frac{j^{2\ell}}{\ell!} (j^2 -
      \ell + 1) \\
      & \le 1 + e^{-j^2} \left( \sum_{\ell=0}^{j^2} \frac{j^{2 (\ell +
            1)}}{\ell!}  - \sum_{\ell = 1}^{j^2}
        \frac{j^{2\ell}}{(\ell - 1)!}
      \right)\\
      & = 1 + e^{-j^2} \frac{j^{2(j^2 + 1)}}{(j^2)!} \le 1 +
      \frac{j}{\sqrt{2\pi}},
    \end{split}\end{equation*}
  and
  \begin{equation*}\begin{split}
      \sum_{k=(j+1)^2}^{n-1}  \frac{1}{k!} \int_0^{(j+1)^2} s^k e^{-s}\
      ds &= e^{-(j+1)^2}
      \sum_{k=(j+1)^2}^{n-1}\sum_{\ell=k+1}^{\infty}\frac{(j+1)^{2\ell}}{\ell!}
      \\
      &= 
e^{-(j+1)^2}\sum_{\ell=(j+1)^2+1}^\infty
      \frac{(j+1)^{2\ell}}{\ell!} \left(\ell -
        (j+1)^2\right) \\
      & = e^{-(j+1)^2}\left(\sum_{\ell=(j+1)^2+1}^\infty
        \frac{(j+1)^{2\ell}}{(\ell-1)!} - \sum_{\ell = (j+1)^2+ 1}^\infty
        \frac{(j+1)^{2(\ell + 1)}}{\ell!} \right) \\
      & =  e^{-(j+1)^2} \frac{(j+1)^{2((j+1)^2 + 1)}}{(j + 1)^2!}
      \le \frac{j+1}{\sqrt{2\pi}}.
    \end{split}
  \end{equation*}
  
  The second integral in \eqref{E:variance} is equal to 
  \begin{equation*}\begin{split}
      \frac{1}{\pi^2} & \sum_{k,\ell=0}^{n-1} \frac{1}{k! \ell!}
      \int_0^j r^{k+\ell + 1} e^{-r^2} \ dr \int_j^{j+1} r^{k+\ell+1}
      e^{-r^2} \ dr \int_0^{2\pi} e^{i\varphi (k-\ell)} \ d\varphi
      \int_\theta^{2\pi} e^{i\varphi(\ell-k)}\ d\varphi \\
      &= \left(1-\frac{\theta}{2\pi}\right) \sum_{k=0}^{n-1}
      \left(\frac{1}{k!} \int_0^{j^2} s^k e^{-s} \ ds \right)
      \left(\frac{1}{k!} \int_{j^2}^{(j+1)^2} s^k e^{-s} \ ds \right) \\
      & \le \left(1-\frac{\theta}{2\pi}\right) \sum_{k=0}^{n-1}
      \frac{1}{k!} \int_{j^2}^{(j+1)^2} s^k e^{-s} \ ds
    \end{split}\end{equation*}
  since the first angular integral is nonzero only for $k=\ell$, and
  the third integral in \eqref{E:variance} is similarly bounded by
  \[
  \frac{\theta}{2\pi} \sum_{k=0}^{n-1} \frac{1}{k!}
  \int_{j^2}^{(j+1)^2} s^k e^{-s} \ ds.
  \]
  The
  function $s \mapsto s^k e^{-s}$ is unimodal for $s > 0$ and takes on
  its maximum value at $s=k$, so
  \[
  \int_{j^2}^{(j+1)^2} s^k e^{-s} \ ds
  \le \begin{cases} (2j+1)j^{2k}e^{-j^2} & \text{when } k \le j^2,\\
     (2j+1)(j+1)^{2k}e^{-(j+1)^2} & \text{when } k \ge (j+1)^2, \text{
       and} \\
     k! & \text{always.}
     \end{cases}
  \]
  From this it follows that the sum of the second and third integrals in
  \eqref{E:variance} is bounded by 
  \begin{equation} \label{E:small-gamma-sum}
  \sum_{k=0}^{n-1} \frac{1}{k!}\int_{j^2}^{(j+1)^2} s^k e^{-s} \
    ds \le 3(2j+1).
  \end{equation}

  The final integral in \eqref{E:variance} is equal to
  \begin{equation} \label{E:final-integral}
    \begin{split}
      \frac{1}{\pi^2} & \sum_{k,\ell=0}^{n-1} \frac{1}{k! \ell!}
      \left(\int_j^{j+1} r^{k+\ell + 1} e^{-r^2} \ dr\right)^2
      \int_0^\theta e^{i\varphi (k-\ell)} \ d\varphi
      \int_\theta^{2\pi} e^{i\varphi(\ell-k)}\ d\varphi.
    \end{split}\end{equation}
  For $k \neq \ell$,
  \begin{equation*}
    \begin{split}
      \int_\theta^{2\pi} e^{i\varphi(\ell-k)}\ d\varphi
      & = - \int_0^\theta e^{i\varphi(\ell - k)} \ d\varphi
      = - \overline{\int_0^\theta e^{i\varphi(k - \ell )} \ d\varphi}
    \end{split}
  \end{equation*}
  so each summand in \eqref{E:final-integral} with $k \neq \ell$ is
  negative.  Thus \eqref{E:final-integral} is bounded by
  \begin{equation*}
      \frac{\theta(2\pi - \theta)}{\pi^2} \sum_{k=0}^{n-1}
      \left(\frac{1}{k!} \int_j^{j+1} r^{2k + 1} e^{-r^2} \ dr\right)^2
      = \frac{\theta}{2\pi} \left(1 - \frac{\theta}{2\pi}\right)
      \sum_{k=0}^{n-1}
      \left(\frac{1}{k!} \int_{j^2}^{(j+1)^2} s^k e^{-s} \ ds\right)^2,
  \end{equation*}
  which by \eqref{E:small-gamma-sum} is less than $\frac{3}{4} (2j + 1)$.
\end{proof}

\section{Deviations}
\label{S:deviations}

The goal of this section is to obtain sharp concentration results for
the eigenvalues $\lambda_k$ about their predicted locations
$\tilde{\lambda}_k$.  Recall that we only defined $\tilde{\lambda}_k$
for a restricted range of $k$; for the outermost eigenvalues, for
which we did not define $\tilde{\lambda}_k$, we will make use of the
following sloppy estimate.
\begin{lemma} 
  \label{T:quick-and-dirty} 
  For each $k$ and any random variable $\alpha \in \C$ with
  $\abs{\alpha} \le 1$,
  \[
  \E \abs{\lambda_k - \alpha}^p \le 4^p +
  \left(\frac{4}{3}\right)^{p-1}\left(\frac{2}{n}\right)^{\frac{p}{2}}
  \Gamma \left(1 + \frac{p}{2} \right).
  \]
  and
  \[
  \Prob \left[ \abs{\lambda_k - \alpha} \ge t \right]
  \le e^{-nt^2 / 4}
  \]
  for $t > 4$.
\end{lemma}

\begin{proof}
  For any $t \ge 1$,
  \begin{equation}  \label{E:quick}
    \Prob \left[ \abs{\lambda_k - \alpha} \ge t \right]
    \le \Prob \left[ \abs{\lambda_k} \ge t-1 \right]
    \le \Prob \left[ \n( \C \setminus (t-1)D ) \ge 1
    \right]
    \le \E \n (\C \setminus (t-1) D), 
  \end{equation}
  by Markov's inequality. Proposition \ref{T:outside-disc} implies
  that for any $R > 1$,
  \begin{equation*}
      \E \n (\C \setminus RD)
      \le \exp \left[ - \left(R^2 - \frac{3}{2} - 2 \log R \right) n \right],
  \end{equation*}
  since $\log n \le n$.  For $R > 3$,
  \[
  R^2 - \frac{3}{2} - 2 \log R > \frac{R^2}{2},
  \]
  and so
  \begin{equation} \label{E:dirty}
    \E \n (\C \setminus RD) \le e^{-n R^2/2}.
  \end{equation}
  Combining \eqref{E:quick} and \eqref{E:dirty}, we obtain
  \begin{equation*} \begin{split}
    \E \abs{\lambda_k - \alpha}^p &= \int_0^\infty p t^{p-1} \Prob\left[
      \abs{\lambda_k - \alpha} \ge t \right] \ dt \\
    & \le \int_0^4 p t^{p-1} \ dt + \int_4^\infty p t^{p-1} e^{-n(t-1)^2/2} \
    dt \\
    &\le 4^p + \left(\frac{4}{3}\right)^{p-1} p  \int_3^\infty s^{p-1} e^{-ns^2/2} \ ds \\
    & \le 4^p + \left(\frac{4}{3}\right)^{p-1}\left(\frac{2}{n}\right)^{\frac{p}{2}} \Gamma
    \left(1 + \frac{p}{2} \right)
    \qedhere
  \end{split}
  \end{equation*}
  and
  \[
  \Prob \left[ \abs{\lambda_k - \alpha} \ge t \right]
  \le e^{-n(t-1)^2 / 2} \le e^{-nt^2 / 4}
  \]
  for $t > 4$.
\end{proof}

We will need stronger concentration for most of the eigenvalues, which
we get as a consequence of the following.

\begin{prop}
  \label{T:counting-function-deviation}
  For each $1 \le j \le \sqrt{n-1} - 1$, $0\le \theta \le 2\pi$, and
  $t > 0$,
  \[
  \Prob \left[ \n(A_{j,\theta}) - \frac{n \abs{A_{j,\theta}}}{\pi} \ge
    t \right] \le \exp \left[-\min \left\{\frac{t^2}{64 j},
      \frac{t}{2} \right\} \right].
  \]
  If $j \le \sqrt{n} - \sqrt{\log n} - 1$, then
  \[
  \Prob \left[ \frac{n \abs{A_{j,\theta}}}{\pi} - \n(A_{j,\theta}) \ge
    t \right] \le 3 \exp \left[-\min \left\{\frac{t^2}{256 j},
      \frac{t}{4} \right\} \right].
  \]
\end{prop}

\begin{proof}
  The first claim follows immediately from Propositions
  \ref{T:ew-counting-deviations}, \ref{T:means}, and
  \ref{T:variances}.  For the second, the assumption on $j$ implies
  that $A_{j,\theta} \subseteq \left(1 - \sqrt{\frac{\log
        n}{n}}\right) D$, and so by Propositions
  \ref{T:ew-counting-deviations}, \ref{T:means}, and
  \ref{T:variances},
  \begin{equation*}
    \begin{split}
      \Prob \left[ \frac{n \abs{A_{j,\theta}}}{\pi} - \n(A_{j,\theta}) \ge t \right]
      & \le \Prob \left[ \E \n(A_{j,\theta})-
        \n(A_{j,\theta}) 
        \ge t - e^2 \right] \\
      &\le \exp \left[-\min \left\{\frac{(t-e^2)^2}{64
            j}, \frac{t-e^2}{2} \right\} \right]
    \end{split}
  \end{equation*}
  for $t > e^2$.  If $t \ge 2e^2$, then $t-e^2 \ge t/2$, so 
  \[
  \Prob \left[ \E \n(A_{j,\theta}) - \n(A_{j,\theta}) \ge t \right]
      \le \exp \left[-\min \left\{\frac{t^2}{256
            j}, \frac{t}{4} \right\} \right].
  \]
  On the other hand, if $t < 2 e^2$, then 
  \[
  \exp \left[-\min \left\{\frac{t^2}{256 j}, \frac{t}{4} \right\}
  \right] > e^{-e^4/64} > 1/3,
  \]
  which implies the second claim.
\end{proof}

The concentration inequalities for the $\n(A_{j,\theta})$ together
with geometric arguments yield the following concentration for
individual eigenvalues.

\begin{thm}
  \label{T:eigenvalue-deviation}
  There are constants $C, c > 0$ such that for those $k$ with
  $\ell=\lceil\sqrt{k}\rceil\le \sqrt{n}-\sqrt{\log n }$,
  \begin{itemize}
  \item when $9\le s\le \pi(\ell-1)+2$,
    \[
    \Prob\left[\abs{\lambda_k-\tilde{\lambda}_k} >
      \frac{s}{\sqrt{n}}\right] \le C \exp\left[-\min \left\{
        \frac{(s-9)^2}{ 256\pi^2(\ell - 1)},
        \frac{s-9}{4\pi}\right\}\right];
    \]
    
  \item  when $s>\pi(\ell-1)+2,$
    
    \[
    \Prob\left[\abs{\lambda_k-\tilde{\lambda}_k} > \frac{s}{\sqrt{n}}\right]\le
    Ce^{-cs^2}.
    \]
    
  \end{itemize}
\end{thm}

\begin{proof}
  Trivially,
  \[
  \Prob \left[ \abs{\lambda_k - \tilde{\lambda}_k} \ge t\right]
  = \Prob \left[  \abs{\lambda_k - \tilde{\lambda}_k} \ge t \text{ and
    } \lambda_k \prec \tilde{\lambda}_k \right]
  + \Prob \left[  \abs{\lambda_k - \tilde{\lambda}_k} \ge t \text{ and
    } \lambda_k \succ \tilde{\lambda}_k \right].
  \]
  
  \medskip
  
  \noindent\textbf{Case 1: $\lambda_k \prec \tilde{\lambda}_k$.}
  
  \medskip
  
  With probability $1$, $\lambda_k \prec \tilde{\lambda}_k$ implies
  that either
  \[
  \frac{\ell-1}{\sqrt{n}}\le\abs{\lambda_k}<  \frac{\ell}{\sqrt{n}}
  \qquad and \qquad \arg \lambda_k < \arg \tilde{\lambda}_k = \frac{2 \pi q}{2 \ell - 1}.
  \]
  or
  \[
  \abs{\lambda_k} < \abs{\tilde{\lambda}_k} = \frac{\ell - 1}{\sqrt{n}}.
  \]
  Observe that, in either case, $\abs{\lambda_k - \tilde{\lambda}_k} < \frac{2\ell -
    1}{\sqrt{n}}$.  
  
  If $\abs{\lambda_k - \tilde{\lambda}_k} \ge s/\sqrt{n}$ and
  $a(\theta,\varphi)$ denotes the length of the shorter arc on the
  unit circle between $e^{i\theta}$ and $e^{i\varphi}$, then the
  elementary estimate
  \[
  \abs{Re^{i \theta} - re^{i \varphi}} \le r \cdot a(\theta,\varphi) + \abs{R-r},
  \]
  implies that when $\abs{\lambda_k} \in \left[\frac{\ell-2}{\sqrt{n}},
  \frac{\ell}{\sqrt{n}}\right)$ and $s\ge 1$,
  \begin{equation}\label{E:angle-limit-1}
    a\left(\arg \lambda_k, \frac{2\pi q}{2\ell-1}\right)\ge
    \frac{s-1}{\ell - 1}.
  \end{equation}

  Suppose that $\frac{s-1}{\ell - 1} < \frac{2\pi q}{2\ell-1}$.  Since
  either $\abs{\lambda_k} < \frac{\ell-1}{\sqrt{n}}$ or
  $\arg{\lambda_k} < \frac{2\pi q}{2\ell - 1}$, Inequality
  \eqref{E:angle-limit-1} implies that
  \[
  \lambda_k \prec \frac{\ell - 1}{\sqrt{n}} \exp\left[ i \left( \frac{2\pi q}{2\ell-1} -
    \frac{s-1}{\ell - 1}\right)\right],
  \]
  and so
  \[
  \n\bigl(A_{\ell - 1, \frac{2\pi q}{2\ell-1} -
    \frac{s-1}{\ell - 1}}\bigr) \ge k = (\ell - 1)^2 + q.
  \]
  Since
  \[
  \frac{n}{\pi} \abs{A_{j,\theta}} = j^2 + \frac{\theta}{2\pi} (2j + 1),
  \]
  we have
  \[
  \frac{n}{\pi} \abs{A_{\ell - 1, \frac{2\pi q}{2\ell-1} -
          \frac{s-1}{\ell - 1}}} = k - \frac{2 \ell - 1}{2\pi (\ell -
          1)} (s-1) \le k - \frac{s-1}{\pi},
  \]
  and so Proposition \ref{T:counting-function-deviation} implies that
  \begin{equation*}
    \begin{split}
      \Prob \left[\n\bigl(A_{\ell - 1, \frac{2\pi q}{2\ell-1} -
          \frac{s-1}{\ell - 1}}\bigr) \ge k\right]
      & \le \Prob \left[ \n\bigl(A_{\ell - 1, \frac{2\pi q}{2\ell-1} -
          \frac{s-1}{\ell - 1}}\bigr) - \frac{n}{\pi} \abs{A_{\ell - 1, \frac{2\pi q}{2\ell-1} -
            \frac{s-1}{\ell - 1}}} \ge \frac{s-1}{\pi} \right] \\
      & \le \exp\left[-\min \left\{ \frac{(s-1)^2}{64 \pi^2 (\ell -
            1)}, \frac{s-1}{2 \pi}\right\}\right].
    \end{split}
  \end{equation*}
  
  Now suppose that $\frac{2\pi q}{2\ell-1} \le \frac{s-1}{\ell - 1}
  \le \pi $ (note that $\frac{s-1}{\ell-1}$ is a lower bound for the
  length of a shortest path on the circle, hence the upper bound of
  $\pi$, and that the interval in question is non-empty only if
  $q\le\frac{2\ell-1}{2}$).  Then
  \[
  \lambda_k \prec \frac{\ell - 2}{\sqrt{n}} \exp\left[ i \left( 2\pi +
      \frac{2\pi q}{2\ell-1} - \frac{s-1}{\ell - 1}\right)\right],
  \]
  and so
  \[
  \n\bigl(A_{\ell - 2, 2\pi + \frac{2\pi q}{2\ell-1} -
    \frac{s-1}{\ell - 1}}\bigr) \ge k.
  \]
  Now
  \[
  \frac{n}{\pi} \abs{A_{\ell - 2, 2\pi + \frac{2\pi q}{2\ell-1} -
      \frac{s-1}{\ell - 1}}} 
  \le k - \frac{2\ell - 3}{2\pi(\ell - 1)} (s-1)
  \le k - \frac{s-1}{2\pi}
  \]
  for $\ell \ge 2$. Thus in this range of $s$ Proposition
  \ref{T:counting-function-deviation} implies that
  \begin{equation*}
    \begin{split}
      \Prob \left[\n\bigl(A_{\ell - 2, 2\pi - \frac{2\pi q}{2\ell-1} -
    \frac{s-1}{\ell - 1}} \bigr) \ge k\right]
      & \le \Prob \left[ \n\bigl(A_{\ell - 2, 2\pi - \frac{2\pi q}{2\ell-1} -
    \frac{s-1}{\ell - 1}} \bigr) - \frac{n}{\pi} \abs{A_{\ell - 2, 2\pi - \frac{2\pi q}{2\ell-1} -
    \frac{s-1}{\ell - 1}} } \ge \frac{s-1}{2\pi} \right] \\
      & \le \exp\left[-\min \left\{ \frac{(s-1)^2}{256 \pi^2 (\ell -
            1)}, \frac{s-1}{4 \pi}\right\}\right].
    \end{split}
  \end{equation*}
 
  As observed above, the estimates above cover the entire possible
  range of $s$, and so
  \[
  \Prob \left[  \abs{\lambda_k - \tilde{\lambda}_k} \ge \frac{s}{\sqrt{n}} \text{ and
    } \lambda_k \prec \tilde{\lambda}_k \right]
  \le \exp\left[-\min \left\{ \frac{(s-1)^2}{256 \pi^2 (\ell -
            1)}, \frac{s-1}{4 \pi}\right\}\right]
  \]
  for all $s \ge 1$.

\medskip

\noindent \textbf{Case 2: $\lambda_k \succ \tilde{\lambda}_k$.}

\medskip

  With probability $1$, $\lambda_k \succ \tilde{\lambda}_k$
  implies that either
  \[
  \frac{\ell-1}{\sqrt{n}} \le \abs{\lambda_k} < \frac{\ell}{\sqrt{n}}
  \quad \text{and} \quad
  \arg \lambda_k > \arg \tilde{\lambda}_k = \frac{2 \pi q}{2 \ell - 1}.
  \]
  or
  \[
  \abs{\lambda_k} \ge \frac{\ell}{\sqrt{n}}.
  \]
  Observe that if $\abs{\lambda_k - \tilde{\lambda}_k} \ge s/\sqrt{n}$
  and $\abs{\lambda_k} \in \left[\frac{\ell-1}{\sqrt{n}},
    \frac{\ell+1}{\sqrt{n}}\right)$, then as above,
  \begin{equation}\label{E:angle-limit-2}
    a\left(\arg \lambda_k, \frac{2\pi q}{2\ell-1}\right)\ge    \frac{s-2}{\ell - 1}
  \end{equation}
  for $s \ge 2$.  We will need to make different arguments depending
  on the value of $\frac{s-2}{\ell - 1}$.
  
  \begin{enumerate}
  \item Suppose first that $\frac{s-2}{\ell - 1} <
    2\pi  - \frac{2\pi q}{2\ell - 1}$. 
    
    Since either $\abs{\lambda_k} \ge \frac{\ell}{\sqrt{n}}$ or $\arg
    \lambda_k > \frac{2\pi q}{2\ell - 1}$, Inequality
    \eqref{E:angle-limit-2} implies that
    \[
    \lambda_k \succ \frac{\ell - 1}{\sqrt{n}} \exp \left[ i
      \left(\frac{2 \pi q}{2\ell - 1}
        + \frac{s-2}{\ell - 1}\right) \right],
    \]
    and so
    \[
    \n\bigl(A_{\ell - 1, \frac{2\pi q}{2\ell-1} +
      \frac{s-2}{\ell - 1}}\bigr) < k.  
    \]
    Since
    \[
    \frac{n}{\pi} \abs{A_{\ell - 1, \frac{2\pi q}{2\ell-1} +
        \frac{s-2}{\ell - 1}}} = k + \frac{2 \ell - 1}{2\pi (\ell -
      1)} (s-2) \ge k + \frac{s-2}{\pi},
    \]
    Proposition \ref{T:counting-function-deviation} implies that
    \begin{equation*}
      \begin{split}
        \Prob \left[\n\bigl(A_{\ell - 1, \frac{2\pi q}{2\ell-1} +
            \frac{s-2}{\ell - 1}}\bigr) < k\right]
        & \le \Prob \left[\frac{n}{\pi} \abs{A_{\ell - 1, \frac{2\pi q}{2\ell-1} +
              \frac{s-2}{\ell - 1}}} - \n\bigl(A_{\ell - 1, \frac{2\pi q}{2\ell-1} -
            \frac{s-1}{\ell - 1}}\bigr) > \frac{s-2}{\pi} \right] \\
        & \le 3 \exp\left[-\min \left\{ \frac{(s-2)^2}{64 \pi^2 (\ell -
              1)}, \frac{s-2}{2 \pi}\right\}\right].
      \end{split}
    \end{equation*}
    
  \item Next suppose that $2\pi - \frac{2\pi q}{2\ell-1} \le
    \frac{s-2}{\ell - 1} \le \pi $ (note that this case only occurs
    when $q\ge\frac{2\ell-1}{2}$).  Then we have that
    \[
    \lambda_k \succ \frac{\ell}{\sqrt{n}} \exp\left[ i \left(
        \frac{2\pi q}{2\ell-1} + \frac{s-2}{\ell - 1} - 2\pi\right)\right],
    \]
    and so
    \[
    \n\bigl(A_{\ell, \frac{2\pi q}{2\ell-1} +
      \frac{s-2}{\ell - 1} - 2\pi} \bigr) < k.
    \]
    Now
    \[
    \frac{n}{\pi} \abs{A_{\ell, \frac{2\pi q}{2\ell-1} -
        \frac{s-2}{\ell - 1} - 2\pi}} 
    \ge k + \frac{2\ell + 1}{2\pi(\ell - 1)} (s-2) - 2
    \ge k + \frac{s-2}{\pi} - 2
    \ge k + \frac{s-9}{\pi}
    \]
    for $s \ge 9$. Thus in this range Proposition
    \ref{T:counting-function-deviation} implies that
    \begin{equation*}
      \begin{split}
        \Prob \left[\n\bigl(A_{\ell, \frac{2\pi q}{2\ell-1} +
            \frac{s-2}{\ell - 1} - 2\pi} \bigr) < k\right]
        & \le \Prob \left[ \frac{n}{\pi} \abs{A_{\ell, \frac{2\pi q}{2\ell-1} +
              \frac{s-2}{\ell - 1} - 2\pi}  } - \n\bigl(A_{\ell, \frac{2\pi q}{2\ell-1} +
            \frac{s-2}{\ell - 1} - 2\pi} \bigr)  > \frac{s-9}{\pi} \right] \\
        & \le 3 \exp\left[-\min \left\{ \frac{(s-9)^2}{64 \pi^2 (\ell -
              1)}, \frac{s-9}{2 \pi}\right\}\right].
      \end{split}
    \end{equation*}
    
  \item Now suppose that $\pi< \frac{s-2}{\ell - 1}\le
    \frac{\sqrt{n}-\sqrt{\log n}+\ell-3}{\ell-1} $. 
    
    By the triangle inequality, $\abs{\lambda_k} \ge
    \frac{s}{\sqrt{n}}-\abs{\tilde{\lambda}_k} = \frac{s - \ell +
      1}{\sqrt{n}}$, and so
    \[
    \n\left(\frac{s - \ell + 1}{\sqrt{n}} D \right) 
    = \n \left(A_{s-\ell, 2\pi}\right)
    < k.
    \]
    The inequality $\frac{s-2}{\ell - 1}\le \frac{\sqrt{n}-\sqrt{\log
        n}+\ell-3}{\ell-1} $ is equivalent to
    $s-\ell\le\sqrt{n}-\sqrt{\log n}-1$, and so the second estimate of
    Proposition \ref{T:counting-function-deviation} applies.  Since
    $k=(\ell-1)^2+q$ and $1\le q\le 2\ell-1\le \frac{2(s-2)}{\pi}+1$,
    \begin{equation*}
      \begin{split}
        \frac{n}{\pi}\abs{A_{s-\ell, 2\pi}} &=(s-\ell+1)^2\\
        &=s^2-2s(\ell-1)+k-q\ge
        s^2\left(1-\frac{2}{\pi}\right)-\frac{6s}{\pi}-\frac{4}{\pi}-1+k,
      \end{split}
    \end{equation*}
    and so
    \begin{equation*}
      \begin{split}
        \Prob \left[\n\bigl(A_{s-\ell, 2\pi} \bigr) <
          k\right]&\le\Prob\left[\frac{n}{\pi}\abs{A_{s-\ell,
              2\pi}}-\n\bigl(A_{s-\ell, 2\pi}
          \bigr)>\frac{n}{\pi}\abs{A_{s-\ell,
              2\pi}}-k\right]\\
        &\le \Prob\left[\frac{n}{\pi}\abs{A_{s-\ell,
              2\pi}}-\n\bigl(A_{s-\ell, 2\pi}
          \bigr)>s^2\left(1-\frac{2}{\pi}\right)+\frac{6s}{\pi}-\frac{4}{\pi}-1\right]\\&\le
        3 \exp \left[-\min \left\{\frac{s^3}{2304}, \frac{s^2}{12}
          \right\} \right],
      \end{split}\end{equation*}
    for $s\ge 9$.
  \item Finally, suppose that $\frac{s-2}{\ell - 1}>
    \frac{\sqrt{n}-\sqrt{\log n}+\ell-3}{\ell-1} $; that is, that
    $s-\ell>\sqrt{n}-\sqrt{\log n}-1$.  As in the previous case,
    $\lambda \succ \tilde{\lambda}_k$ and
    $\abs{\lambda_k-\tilde{\lambda}_k}>\frac{s}{\sqrt{n}}$ implies
    that
    \[
    \n\left(\frac{s - \ell + 1}{\sqrt{n}} D \right) 
    = \n \left(A_{s-\ell, 2\pi}\right)
    < k,
    \]
    but the second inequality of Proposition
    \ref{T:counting-function-deviation} does not apply.  If
    $\frac{s-\ell+1}{\sqrt{n}}\ge 1$, then $\Prob\left[\n
      \left(A_{s-\ell, 2\pi}\right) < k\right]=0$; otherwise, one can
    use the weaker estimate of Proposition \ref{T:means} for $\E\n
    \left(A_{s-\ell, 2\pi}\right)$ to get that
    \begin{equation*}\begin{split}
        \Prob\left[\n \left(A_{s-\ell, 2\pi}\right)
          < k\right]&\le\Prob\left[\E\n \left(A_{s-\ell, 2\pi}\right)-\n \left(A_{s-\ell, 2\pi}\right)>
          \frac{n}{\pi}\abs{A_{s-\ell, 2\pi}}-e\sqrt{n} -k\right]\\
        &\le \Prob\left[\E \n\bigl(A_{s-\ell, 2\pi} \bigr)-\n\bigl(A_{s-\ell, 2\pi} \bigr)>s^2\left(1-\frac{2}{\pi}\right)+\frac{6s}{\pi}-\frac{4}{\pi}-1-e\sqrt{n}\right]
      \end{split}\end{equation*}
    Since $s\ge \sqrt{n}-\sqrt{\log n}$, the lower bound above can be
    replaced, for $n$ large enough, by $cs^2$ for any
    $c<\left(1-\frac{2}{\pi}\right)$.  Applying Bernstein's inequality
    and the variance estimate of Proposition \ref{T:variances} then
    yields
    \[
    \Prob\left[\n \left(A_{s-\ell, 2\pi}\right)
      < k\right]\le
    C\exp\left[-\min\left\{c^2s^3,\frac{cs^2}{2}\right\}\right].
    \qedhere
    \]
\end{enumerate}
\end{proof}

\section{Distances in the circular law}
\label{S:distances}

In this section, we assemble the previous results to give quantitative
versions of the circular law.  We first note that our estimates for
the means of the eigenvalue counting functions for balls already yield
the correct order for the total variation distance between the
averaged empirical spectral measure and the uniform measure on the
disc.  The fact that the mean spectral measure $\E \mu_n$ converges to
the uniform measure $\nu$ in total variation can be deduced from
Mehta's work \cite[Chapter 15]{Mehta-3}.  We would not be surprised to
learn that the correct rate of convergence is known, but we have not
found it in the literature.

\begin{prop}
  \label{T:TV-mean}
  For each positive integer $n$, $\frac{1}{e \sqrt{n}} \le d_{TV}
  (\nu, \E \mu_n) \le \frac{e}{\sqrt{n}}$.
\end{prop}

\begin{proof}
  For any Borel set $A \subseteq \C$, Proposition \ref{T:means}
  implies that
  \[
  \nu(A) - \frac{e}{\sqrt{n}} \le \E \mu_n(A \cap D) \le \nu(A),
  \]
  so
  \[
  \nu(A) - \E\mu_n(A) \le \nu(A) - \E \mu_n(A
  \cap D) \le \frac{e}{\sqrt{n}}.
  \]
  Furthermore,
  \[
  \E \mu_n(\C \setminus D) = 1 - \E \mu_n(D) = \nu(D) - \E \mu_n(D)
  \le \frac{e}{\sqrt{n}},
  \]
  so
  \[
  \E \mu_n(A) - \nu(A) \le \E \mu_n(\C \setminus D) + \E \mu_n(A \cap
  D) - \nu(A) \le \frac{e}{\sqrt{n}}.
  \]
  Thus $d_{TV}(\nu, \E \mu_n) = \sup_A \abs{\nu(A) - \E \mu_n(A)} \le
  \frac{e}{\sqrt{n}}$.

  On the other hand, the proof of Proposition \ref{T:outside-disc}
  implies that
  \[
  \E \mu_n(\C \setminus D) = \frac{e^{-n} n^n}{n!} \ge
  \frac{1}{e\sqrt{n}}
  \]
  by Stirling's approximation.  Since $\nu(\C \setminus D) = 0$, this
  provides the lower bound.
\end{proof}

The deviation estimates in the previous section allow us to finally
establish the stronger version of the circular law given in Theorem
\ref{T:expected-Wp-distance-intro}, via the following proposition.

\begin{prop}
  \label{T:coupling-to-uniform}
  For any positive integers $m \le n$ and any $p \ge 1$, 
  \[
  \E W_p(\mu_n, \nu) < \frac{8}{\sqrt{n}} + 2 \max_{1\le k \le n-m}
  \left( \E \abs{\lambda_k - \tilde{\lambda}_k}^p \right)^{1/p} + C
  \left(4+\sqrt{\frac{p}{n}}\right) \left(\frac{m}{n}\right)^{1/p}.
  \]
\end{prop}

\begin{proof}
  By Lemma \ref{T:quick-and-dirty}, 
  \begin{equation*}
    \begin{split}
      \E W_p(\mu_n, \nu_n)^p & \le \frac{1}{n} \left( \sum_{k = 1}^{n-m} \E \abs{\lambda_k -
          \tilde{\lambda}_k}^p + \sum_{k=n-m+1}^n \E \abs{\lambda_k -
          u}^p \right) \\
      & \le \max_{1\le k \le n-m} \E \abs{\lambda_k -
        \tilde{\lambda}_k}^p + \left(4^p+ \left(\frac{32}{9n}\right)^{\frac{p}{2}}\Gamma\left(1+ \frac{p}{2}\right) \right)\frac{m}{n},
    \end{split}
  \end{equation*}
  where $u$ is uniform in the outer part of the disc and independent
  of $\lambda_k$. Lemma \ref{T:coupling} and the triangle inequality
  for $W_p$ imply that
  \[
  \E W_p(\mu_n, \nu) < \frac{8}{\sqrt{n}} + 2
  \max_{1\le k \le n-m} \left( \E \abs{\lambda_k - \tilde{\lambda}_k}^p
  \right)^{1/p}
  + \left(4+\sqrt{\frac{32}{9n}} \Gamma\left(1+ \frac{p}{2}\right)^{1/p}\right) \left(\frac{m}{n}\right)^{1/p},
  \]
  and the proposition follows from Stirling's approximation.
\end{proof}

\begin{proof}[Proof of Theorem \ref{T:expected-Wp-distance-intro}]
  Let $m$ be such that $n-m$ is a perfect square, and such that if
  $1\le k\le n-m$, then
  $\ell = \lceil\sqrt{k}\rceil\le\sqrt{n}-\sqrt{\log n}.$ By Fubini's
  theorem and Corollary \ref{T:eigenvalue-deviation}, for $1\le k\le
  n-m$,
  \begin{equation*}\begin{split}
      \E\abs{\lambda_k-\tilde{\lambda}_k}^p&=\int_0^\infty
      pt^{p-1}\Prob\left[|\lambda_k-\tilde{\lambda}_k|>t\right]dt\\
      &=
      \frac{p}{n^{p/2}}\int_0^\infty
      s^{p-1}\Prob\left[\abs{\lambda_k-\tilde{\lambda}_k} > \frac{s}{\sqrt{n}}\right]ds\\
      &\le\frac{p}{n^{p/2}}\left(\frac{9^p}{p}
        +\int_9^{2+\pi(\ell-1)}Cs^{p-1}\exp\left[-\frac{(s-9)^2}{256\pi^2(\ell-1)}\right]ds
        +\int_{2+\pi(\ell-1)}^\infty
        Cs^{p-1}e^{-cs^2}ds\right)\\
      &\le\frac{p}{n^{p/2}}\left[\frac{9^p}{p}+\int_0^\infty
        Cs^{p-1}e^{-\frac{cs^2}{\ell-1}}ds\right]\\
      &\le\frac{pC^p(\ell-1)^{p/2}\Gamma\left(\frac{p+1}{2}\right)}{n^{p/2}}\\
      &\le\frac{pC^p\Gamma\left(\frac{p+1}{2}\right)}{n^{p/4}},
    \end{split}\end{equation*}
  since $\ell\le\sqrt{n}$.
  
  Noting that we can take $m\le c\sqrt{n\log n}$, it follows from Proposition
  \ref{T:coupling-to-uniform} that
  \[
  \E W_p(\mu_n,\nu) \le \frac{8}{\sqrt{n}} + 2 \frac{C\Gamma
      \left( \frac{p+1}{2} \right)^{\frac{1}{p}}}{n^{1/4}} +
  C\left(4+\sqrt{\frac{p}{n}}\right)\left(\frac{m}{n}\right)^{\frac{1}{p}}
  \le C \max
\left\{\frac{\sqrt{p}}{n^{1/4}}, \left(\frac{\log
        n}{n}\right)^{\frac{1}{2p}}\right\}.
  \]
\end{proof}

\begin{proof}[Proof of Theorem \ref{T:as-rate}]
  By Lemma \ref{T:coupling}, up to the value of absolute constants it
  suffices to prove the theorem with $\nu_n$ in place of $\nu$. Let
  $m$ be as in the proof of Theorem
  \ref{T:expected-Wp-distance-intro}.  For any $t > 0$,
  \begin{equation*}
    \begin{split}
      \Prob \left[ W_p(\mu_n, \nu_n) > t \right]
      & \le \Prob \left[ \frac{1}{n} \sum_{k=1}^n \abs{\lambda_k -
          \tilde{\lambda}_k}^p > t^p\right] \\
      & \le  \Prob\left[\sum_{k=1}^{n-m}\abs{\lambda_k -
          \tilde{\lambda}_k}^p > \frac{nt^p}{2}\right]
      + \Prob\left[\sum_{k=n-m+1}^{n}\abs{\lambda_k -
          \tilde{\lambda}_k}^p > \frac{nt^p}{2}\right] \\
      & \le \sum_{k=1}^{n-m} \Prob\left[ \abs{\lambda_k -
          \tilde{\lambda}_k}^p > \frac{nt^p}{2(n-m)} \right]
      + \sum_{k=n-m+1}^n \Prob\left[ \abs{\lambda_k -
          \tilde{\lambda}_k}^p > \frac{nt^p}{2m} \right] \\
      & \le \sum_{k=1}^{n-m} \Prob\left[ \abs{\lambda_k -
          \tilde{\lambda}_k} > \frac{t}{2} \right]
      + \sum_{k=n-m+1}^n \Prob\left[ \abs{\lambda_k -
          \tilde{\lambda}_k} > \left(\frac{n}{m}\right)^{1/p} \frac{t}{2} \right].
    \end{split}
  \end{equation*}

  Suppose first that $1 \le p \le 2$. If $K > 0$ is large enough, then
  for sufficiently large $n$, Theorem \ref{T:eigenvalue-deviation}
  implies that for $1 \le k \le n-m$,
  \[
  \Prob\left[ \abs{\lambda_k - \tilde{\lambda}_k} > K
    \frac{\sqrt{\log n}}{n^{1/4}} \right] \le \frac{1}{n^3}.
  \]
  Moreover, for $k > n-m$, Lemma \ref{T:quick-and-dirty} implies that
  for sufficiently large $n$,
  \[
  \Prob\left[ \abs{\lambda_k - \tilde{\lambda}_k} > K
    \left(\frac{n}{m}\right)^{1/p} \frac{\sqrt{\log n}}{n^{1/4}}
  \right] \le e^{-cn}.
  \]
  It follows that
  \[
  \sum_{n=1}^\infty \Prob\left[W_p(\mu_n, \nu_n) > 2 K \frac{\sqrt{\log
        n}}{n^{1/4}} \right] < \infty,
  \]
  and an application of the Borel--Cantelli lemma completes the proof.

  Now suppose that $p > 2$. If $K > 0$ is large enough, then similar
  arguments show that 
  \[
  \sum_{n=1}^\infty \Prob\left[W_p(\mu_n, \nu_n) > 2 K
    \left(\frac{\log n}{n} \right)^{1/2p} \right] < \infty,
  \]
and again the proof is completed by applying the Borel--Cantelli lemma.
  Note that the choice of $t = K \left(\frac{\log
      n}{n}\right)^{1/2p}$ is dictated entirely by the fact that the
  tail bound in Lemma \ref{T:quick-and-dirty} only applies when
  \[
  \left(\frac{n}{m}\right)^{1/p} \frac{t}{2} > 4.
  \]
\end{proof}

\section*{Acknowledgements}

The authors thank Nicolas Rougerie and Terry Tao for their comments on
a previous version of this paper.  This research was partially
supported by grants from the Simons Foundation (\#267058 to E.M.\ and
\#264103 to M.M.) and the U.S. National Science Foundation
(DMS-1308725 to E.M.).  This work was carried out while the authors
were visiting the Institut de Math\'ematiques de Toulouse at the
Universit\'e Paul Sabatier; the authors thank them for their generous
hospitality.

\bibliographystyle{plain}
\bibliography{ginibre}

\end{document}